\documentclass[twoside,leqno,11pt]{article}
\usepackage{revstat}

\newcommand{\be}{\begin{equation}} 
\newcommand{\ee}{\end{equation}}
\newcommand{\beq}{\begin{eqnarray}}
\newcommand{\eeq}{\end{eqnarray}}
\newcommand{\nbeq}{\begin{eqnarray*}}
\newcommand{\neeq}{\end{eqnarray*}}

\begin{document}

\title{On Arnold-Villase\~{n}or conjectures for characterizing exponential distribution based on sample of size three
}
\renewcommand{\titleheading}
             {Characterizations of exponential distribution}  
\author{\authoraddress{George P. Yanev}
                      {School of Mathematical and Statistical Sciences,\\
                       The University of Texas Rio Grande Valley,\\
                       U.S.A.
                       \ (george.yanev@utrgv.edu)\\
                       and \\
                       Institute of Mathematics and Informatics,\\
                       Bugarian Academy of Sciences,\\
                       Bulgaria}
}
\renewcommand{\authorheading}
             { George P. Yanev }  

\maketitle

\begin{abstract}
Arnold and Villase\~{n}or \cite{AV13} obtain a series of characterizations of the exponential distribution based on random samples of size two. These results were already applied in constructing goodness-of-fit tests in \cite{JMNOV15}. Extending the techniques from \cite{AV13}, we prove some of Arnold and Villase\~{n}or's conjectures for samples of size three. An example with simulated data is discussed.
\end{abstract}


\begin{ams}
62G30, 62E10.
\end{ams}

\mainpaper  

\section{Introduction}
\label{intro}
In general, the problem of characterization of probability distributions is described as follows. Suppose a family of distributions $\mathcal{F}$ possesses a property $\mathcal{A}$. If, conversely, a distribution has property $\mathcal{A}$ only if it is a member of that family, then property $\mathcal{A}$ characterizes the family $\mathcal{F}$. This result is referred to as a characterization of the distributions in $\mathcal{F}$. Primary motivation for characterizations problems is due to statistical applications. If a statistical procedure assumes that property $\mathcal{A}$ holds, then the underlying distribution must be a member of the family $\mathcal{F}$. Naturally, first characterizations results are for the normal family of distributions. The exponential distribution is one of the non-normal distributions, which has received a lot of attention as well. Comprehensive surveys of exponential characterizations can be found in \cite{A17}, \cite{AH95}, \cite{AV86}, \cite{GKS98}, and \cite{N06}.

More recently, Arnold and Villase\~{n}or \cite{AV13} obtained a series of characterizations of the exponential distribution based on random samples of size two and conjectured possible generalizations for samples of size three. They provide motivation for their results by pointing out an example of a goodness-of-fit construction. A test for exponentiality based on the characterizations in \cite{AV13} was recently constructed in \cite{JMNOV15}. Another possible use of the results in \cite{AV13} and their generalizations, is in verifying modeling assumptions and in simulations (see also \cite{N06}).  Extending the techniques from \cite{AV13}, we will prove some of Arnold and Villase\~{n}or's conjectures. 

Assume throughout that $X_1, X_2$, and $X_3$ are independent random variables with a common absolutely continuous cumulative distribution function (cdf) $F$, such that $F(0)=0$ and probability density function (pdf) $f$. Denote $X_{2:2}:=\max\{X_1, X_2\}$, $X_{3:3}:=\max\{X_1, X_2, X_3\}$, and $\bar{F}=1-F$. Consider the relations:
\be \label{eqn_1}
\sum_{j=1}^3 \frac{1}{j}X_j \quad \mbox{has pdf}\quad \sum_{j=1}^3 {3 \choose j}(-1)^{j-1}jf(jx),
\ee
\be \label{eqn_2}
X_{3:3}\quad \mbox{has pdf}\quad  \sum_{j=1}^3 {3 \choose j}(-1)^{j-1}j\bar{F}(jx), 
\ee
\be  \label{eqn_3}
\sum_{j=1}^3 {3 \choose j}(-1)^{j-1}jf(jx)=\sum_{j=1}^3 {3 \choose j}(-1)^{j-1}j\bar{F}(jx),
\ee
\be  \label{eqn_4}
X_{2:2}+\frac{1}{3}X_3\stackrel{d}{=} X_{3:3} \quad \mbox{and}\quad \sum_{j=1}^3 \frac{1}{j}X_j\stackrel{d}{=} X_{3:3},
\ee
where $\stackrel{d}{=}$ denotes equality in distribution. We will prove, under some regularity assumptions on $F$, that each one of these five conditions, on its own, is sufficient for $X_1, X_2$, and $X_3$ to be exponentially distributed.

\begin{figure}		\includegraphics[width=450pt,height=260pt]{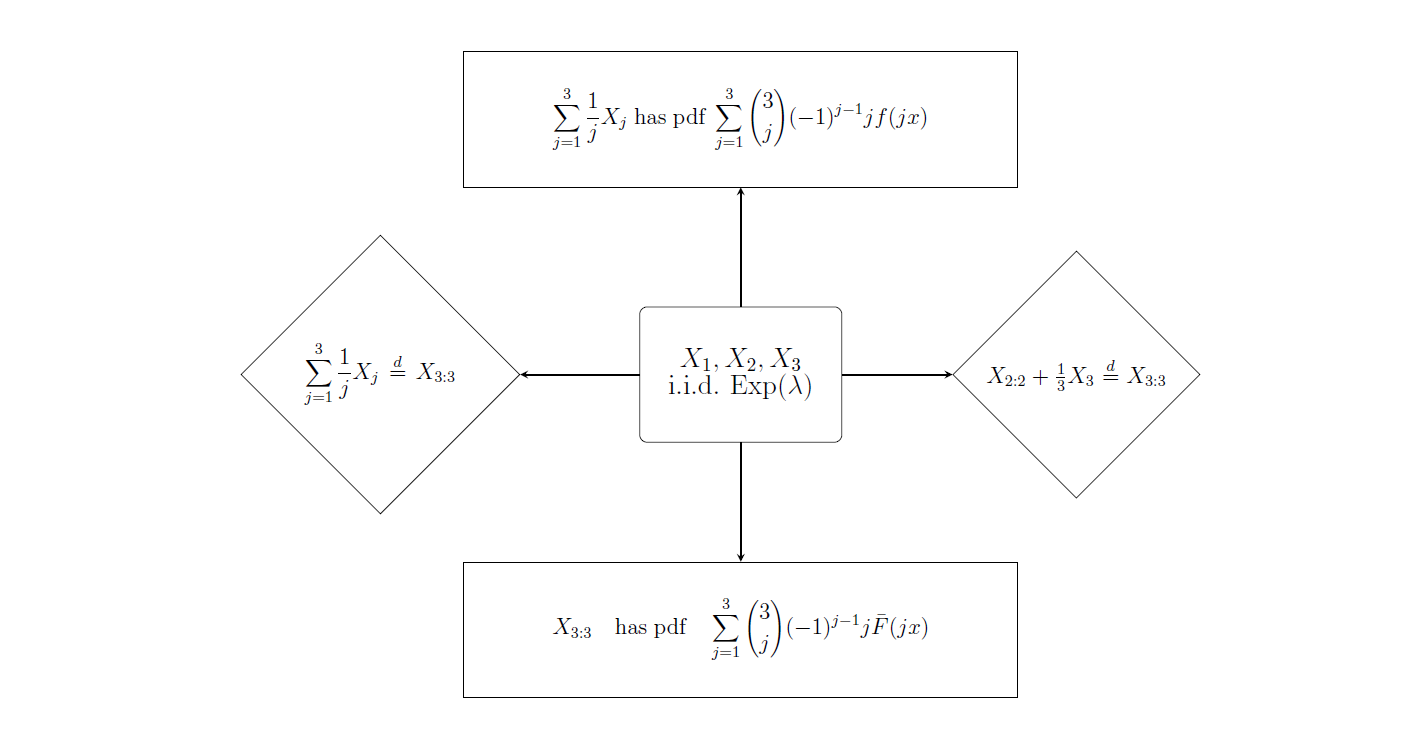}
		\end{figure}
    
We organize this paper as follows. Using Laplace transforms, in Section~\ref{second} we prove the characterization (\ref{eqn_1}). In Section~\ref{third}, we establish characterization (\ref{eqn_2}) utilizing the Taylor series expansion of the cdf $F$. In Section~\ref{fourth}, using a recurrent relation, we prove that (\ref{eqn_3}) is a sufficient condition for having exponential parent. Section~\ref{fifth} contains characterization results based on (\ref{eqn_4}).  In Section~6 we provide an example with simulated data. In the concluding section, we discuss possible extensions of the given results.

\section{Sum of three independent variables}\label{second}
To prove that (\ref{eqn_1}) characterizes the exponential distribution, we will convert it into an equation for the Laplace transform $\varphi(t):=E[e^{-tX_1}]$. 

\begin{theorem}\label{thm1} Assume $\varphi(t)$ is finite for all $t$ in a neighbourhood of zero. If for $x>0$
\be \label{density_14}
\sum_{j=1}^3 \frac{1}{j}X_j \quad \mbox{has pdf}\qquad \sum_{j=1}^3 {3 \choose j} (-1)^{j-1}jf(jx),
\ee
then $X_1\sim \exp(\lambda)$ for some $\lambda>0$.
\end{theorem}

\begin{proof} It follows by (\ref{density_14}), interchanging the order of summation and integration, that 
\beq  \label{s_d_eqn1}
\varphi(t)\varphi\left(\frac{t}{2}\right)\varphi\left(\frac{t}{3}\right) & = & E\left[ 
e^{-t\sum_{j=1}^3 \frac{1}{j}X_j}\right] \\
    & = & \int_0^\infty e^{-tx}\left(\sum_{j=1}^3 {3 \choose j} (-1)^{j-1}jf(jx)\right)\, dx  \nonumber \\
    & = & \sum_{j=1}^3 {3 \choose j} (-1)^{j-1} \int_0^\infty e^{-tx}jf(jx)\, dx  \nonumber \\
    & = & \sum_{j=1}^3 {3 \choose j} (-1)^{j-1}\varphi\left(\frac{t}{j}\right). \nonumber
    \eeq
Dividing  both sides of (\ref{s_d_eqn1}) by $\varphi(t)\varphi(t/2)\varphi(t/3)$, we obtain
\be \label{eqn24}
1=
 \alpha(t)\alpha\left(\frac{t}{2}\right)-3\alpha(t)\alpha\left(\frac{t}{3}\right)+3\alpha\left(\frac{t}{2}\right)\alpha\left(\frac{t}{3}\right),
\ee
where for $t>0$
\be \label{notation4}
\alpha(t):=\frac{1}{\varphi(t)}=\sum_{k=0}^\infty a_kt^k.
\ee
Note that, the series in (\ref{notation4}) is convergent in a neighbourhood of zero, by assumption. 
To prove the theorem, 
it is sufficient to show
that 
\be \label{claim}
\alpha(t)=1+\lambda t, \qquad \lambda>0.
\ee
We will prove 
(\ref{claim}) by calculating the
coefficients of the series in (\ref{notation4}) to be: $a_0=1$, $a_1=\lambda>0$, and $a_k=0$ for $k\ge 2$. It is clear that $a_0=\varphi^{-1}(0)=1$.
Applying Cauchy formula for multiplication of two power series, we have for any nonzero $p$ and $q$,
\be  \label{lemma14}
\alpha\left(\frac{t}{p}
\right)\alpha\left(\frac{t}{q}\right)
 =  \sum_{k=0}^\infty \left( \sum_{j=0}^k \frac{1}{p^jq^{k-j}}a_j a_{k-j}\right) t^k. 
\ee
Now, (\ref{eqn24}) and (\ref{lemma14}) yield for $k\ge 1$
\be \label{coeff}
    \sum_{j=0}^{k}  \left( \frac{1}{2^{k-j}}-\frac{3}{3^{k-j}}+\frac{3}{2^j3^{k-j}}\right)a_ja_{k-j}=0. 
\ee
Setting $k=1$  we see that equation (\ref{coeff}) has as solution any $a_1$. The assumption $F(0)=0$ implies that there is  $\lambda>0$, such that $a_1=\lambda>0$.
If $k=2$, then (\ref{coeff}) yields 
$a_2=0$.
Assuming  $a_{j}=0$ for $2\le j\le k-1$, it follows from (\ref{coeff}) that
\[
\left(1-\frac{1}{2^{k-1}}\right)a_k=0.
\]
Thus, $a_k=0$ for any $k\ge 3$. Therefore,
 (\ref{claim}) holds, which  completes the proof.
 \end{proof}

Note that, conversely,  if $X_i\sim \exp(\lambda)$ for $i=1,2,3$, then (\ref{density_14}) holds true. To show this, it is sufficient to verify (\ref{s_d_eqn1}). Indeed, assuming $X_1\sim \exp(\lambda)$, we have $\varphi(t)=(1+\lambda t)^{-1}$. Therefore,
\nbeq
\int_0^\infty e^{-tx}\left(\sum_{j=1}^3 {3 \choose j} (-1)^{j-1}jf(jx)\right)\, dx & = &  3\varphi(t)-3\varphi\left(\frac{t}{2}\right)+ \varphi\left(\frac{t}{3}\right)\\
& = & \frac{3}{1+\lambda t}-\frac{6}{2+\lambda t}+\frac{3}{3+\lambda t}\\
   & = & \varphi(t)\varphi\left(\frac{t}{2}\right)\varphi\left(\frac{t}{3}\right) \\
   & = & E\left[ e^{-t(X_1+\frac{1}{2}X_2+\frac{1}{3}X_3)}\right], 
    \neeq
which is equivalent to (\ref{density_14}).

\section{Maximum of three independent variables}\label{third}
In this section we will prove that, under some regularity assumptions on $F$,  condition (\ref{eqn_2}) is sufficient for $X_1, X_2$, and $X_3$ to be exponentially distributed. The proof will be based on the Taylor series expansion of $F$.

\begin{theorem}\label{thm2} Assume the cdf  $F$ has a power series representation
for $x$ in a neighborhood of zero. If for $x>0$
\be \label{eqn_21}
X_{3:3}\quad \mbox{has pdf}\quad  \sum_{k=1}^3 { 3 \choose k}(-1)^{k-1}k\bar{F}(kx), 
\ee
then $X_1\sim \exp(1)$.
\end{theorem}
\begin{proof} The relation (\ref{eqn_21}) implies
\be \label{main}
F^2(x)f(x)+F(x)-2F(2x)+F(3x)=0.
\ee
Since 
$
F(x)=\sum_{k=0}^\infty c_kx^k$ and $f(x)=\sum_{k=0}^\infty (k+1)c_{k+1}x^k
$,
Cauchy formula for the product of three power series yields
\be \label{Cauchy}
F^2(x)f(x)=\sum_{k=0}^
\infty \left[\sum_{i=0}^k \sum_{j=0}^i c_jc_{i-j}(k+1-i)c_{k+1-i}\right]x^k.
\ee
Using (\ref{main}) and (\ref{Cauchy}), we obtain for any $k\ge 0$
\be \label{first_d_k}
\sum_{i=0}^k \sum_{j=0}^i c_jc_{i-j}(k+1-i)c_{k+1-i}+c_k(1-2^{k+1}+3^k)=0.
\ee
Since $F(0)=0$, we have $c_0=0$. Also (\ref{first_d_k}) with $k=1$ yields
$c_0^2c_1=0$, which in turn implies that $c_1$ is undetermined. Let us set $c_1=\delta$, where $-\infty<\delta<\infty$. Equation (\ref{first_d_k}) with $k=2$ yields
$c_1^3+2c_2=0$.
Hence, $c_2=\delta^3/2$.
We will prove by induction that 
\be \label{c_k}
c_k=(-1)^{k-1}\frac{\delta^{2k-1}}{k!}
\qquad k=1,2,3,\ldots.
\ee
Indeed, assuming (\ref{c_k}) holds true for $1,2,\ldots, k$, we have
\be  \label{c_k+1}
 \sum_{i=0}^{k+1} \sum_{j=0}^i c_jc_{i-j}(k+2-i)c_{k+2-i}
     =  \sum_{i=2}^{k+1}\sum_{j=1}^{i-1} \frac{(-1)^{k+1}\delta^{2k+1}}{j!(i-j)!(k+1-i)!}.
\ee
Observe  that
\beq \label{induction}
\hspace{1cm} \sum_{i=2}^{k+1}\sum_{j=1}^{i-1} \frac{1}{j!(i-j)!(k+1-i)!}
   & = &
    \sum_{i=2}^{k+1}\frac{1}{i!(k+1-i)!}\sum_{j=1}^{i-1}\frac{i!}{j!(i-j)!}  \\
    \nonumber \\
    & = & 
    \frac{1}{(k+1)!}\sum_{i=2}^{k+1}\frac{(k+1)!}{i!(k+1-i)!}(2^i-2)
    \nonumber \\
    & = & 
     \frac{1}{(k+1)!}\left[\sum_{i=2}^{k+1}{k+1 \choose i}2^i-2\sum_{i=2}^{k+1}{k+1 \choose i}\right]
     \nonumber \\
    & = & 
    \frac{1}{(k+1)!}\left(3^{k+1}-2^{k+2}+1
    \right).
    \nonumber
\eeq
It follows from (\ref{first_d_k}), (\ref{c_k+1}) and (\ref{induction}) that
\[
(-1)^{k+1}\frac{\delta^{2k+1}}{(k+1)!}\left(3^{k+1}-2^{k+2}+1
    \right)+c_{k+1}(1-2^{k+2}+3^{k+1})=0.
\]
Therefore,
\[
c_{k+1}=(-1)^{k}\frac{\delta^{2k+1}}{(k+1)!},
\]
which completes the induction and hence proves (\ref{c_k}). 

Now, we have
\[
F(x) 
     =  \sum_{k=1}^\infty (-1)^{k+1}\frac{\delta^{2k-1}}{k!}x^k 
      =  
     \frac{1}{\delta}\left(1-e^{-\delta^2x}
     \right).
\]
Since $\lim_{x\to \infty}F(x)=1$, we obtain $\delta=1$. The proof is complete.
\end{proof}

It is not difficult to see that, conversely, if $X_1\sim \exp(1)$, then (\ref{eqn_21}) holds. Indeed, under the assumption of unit exponential parent variable,
for the pdf of  $X_{3:3}$ we obtain
\[
3F^2(x)f(x)  =  3(1-e^{-x})^2e^{-x}
      =  3\bar{F}(x)-6\bar{F}(2x)+3\bar{F}(3x),
\]
which is equivalent to (\ref{eqn_21}).

\section{Sums of density and distribution functions}\label{fourth}

In this section we will prove that (\ref{eqn_3}) is a sufficient condition for $X_1$ to be exponentially distributed. It is straightforward that (\ref{eqn_3}) is a necessary condition as well.

\begin{theorem}\label{thm3} Assume that $f$ is right-continuous at zero. If for $x>0$
\be  \label{eqn_31}
\sum_{j=1}^3 {3 \choose j}(-1)^{j-1}jf(jx)=\sum_{j=1}^3 {3 \choose j}(-1)^{j-1}j\bar{F}(jx),
\ee
then  $X_1\sim \exp(1)$.
\end{theorem}

\begin{proof} The relation (\ref{eqn_31}) leads to
\be \label{thm411}
[f(3x)-\bar{F}(3x)]-[f(2x)-\bar{F}(2x)]=[f(2x)-\bar{F}(2x)]-[f(x)-\bar{F}(x)].
\ee
Denoting $Q(y)=f(y)-\bar{F}(y)$, we rewrite (\ref{thm411}) as
\[
Q(y)-Q\left(\frac{2}{3}y\right)=Q\left(\frac{2}{3}y\right)-Q\left(\frac{1}{3}y\right).
\]
Iterating this equation $k$ times and taking limit as $k\to \infty$, we obtain
\[
Q(y)-Q\left(\frac{2}{3}y\right)  =  Q\left(\frac{2}{3}y\right)-Q\left(\frac{1}{3}y\right)
     =  
   \lim_{k\to \infty} Q\left(\left(\frac{2}{3}\right)^ky\right)-Q\left(\left(\frac{1}{3}\right)^ky\right)
     =  0.
\]
This implies $Q(y)=Q(2y/3)$ and thus, 
\be \label{thm412}
Q(y)  =  Q\left(\frac{2}{3}y\right) 
     =  \lim_{k\to \infty} Q\left(\left(\frac{2}{3}\right)^ky\right) 
     =  f(0+)-\bar{F}(0+)
     =  f(0+)-1.
\ee
On the other hand,
\be \label{thm413}
\lim_{y\to \infty}f(y)=\lim_{y\to \infty}f(y)-\lim_{y\to \infty}\bar{F}(y) =
  \lim_{y\to \infty}Q(y)=f(0+)-1.
  \ee
  But since $f$ is integrable, we have $\lim_{y\to \infty}f(y)=0$, and therefore, by (\ref{thm412}) and (\ref{thm413}), $Q(x)=0$. Thus, $f(x)=\bar{F}(x)$ for every $x\ge 0$. This, in turn, implies $X_1\sim \exp(1)$.
  \end{proof}

\section{Sum and maximum of three variables}\label{fifth}

 It is known (e.g., Arnold et al. (2008), p.77) that if $X\sim \exp\{\lambda\}$, then
\be \label{property}
\sum_{j=1}^3 \frac{1}{j}X_j \stackrel{d}{=} X_{3:3} \quad \mbox{and}\quad X_{2:2}+\frac{1}{3}X_3\stackrel{d}{=} X_{3:3}.
\ee
We will prove that both relations in  (\ref{property}) are also characterization properties of the exponential distribution. Next lemma provides the key argument in  the proof of Theorem~1 in \cite{AV13} and of the theorem below.

\begin{lemma}\label{lemma1} If $F(0)=0$, the pdf $f$ has a Taylor series expansion for $x>0$, and
\be \label{lemma}
f^{(m)}(0)=\left[\frac{f'(0)}{f(0)}\right]^{m-1}f'(0), \qquad m=1,2,\ldots,
\ee
then $X_1\sim \exp\{\lambda\}$ for some  $\lambda>0$.https://www.overleaf.com/project/5d377afee73bb23f914b89d6
\end{lemma}

\begin{proof} For the Taylor series of $f(x)$, using (\ref{lemma}), we have for $x>0$
\[ 
f(x)  =  \sum_{m=0}^\infty \frac{f^{(m)}(0)}{m!}x^m 
     =  f(0)+f(0)\sum_{m=1}^\infty \left[\frac{f'(0)}{f(0)}\right]^{m}\frac{x^m}{m!}
     = 
    f(0)\exp\left\{\frac{f'(0)}{f(0)}x\right\}.
\]
Since $f(x)$ is a pdf, we have $f'(0)/f(0)<0$. Denoting $\lambda=-f'(0)/f(0)>0$ and setting $\int_0^\infty f(x)\, dx=1$, we obtain $\lambda=f(0)$. Therefore, $f(x)=\lambda e^{-\lambda x}$.  
\end{proof}

Next theorem can be obtained as a particular case of the results in \cite{YC16}. We include it here since it complements the other results for samples of size three given in Theorems~\ref{thm1}--\ref{thm3} and thus provides an easily reference.

\begin{theorem}\label{thm4} Assume the cdf  $F$ admits a power series representation
in a neighborhood of zero and $F(0)=0$.

(i) If
\be  \label{eqn_main1}
X_{2:2}+\frac{1}{3}X_3\stackrel{d}{=} X_{3:3},
\ee
then  $X_1\sim \exp\{\lambda\}$ for some  $\lambda>0$.

(ii) If
\be  \label{eqn_main2}
\sum_{j=1}^3 \frac{1}{j}X_j \stackrel{d}{=} X_{3:3},
\ee
then  $X_1\sim \exp\{\lambda\}$ for some  $\lambda>0$.
\end{theorem}

\begin{proof} (i).
The pdf of the left-hand side of (\ref{eqn_main1}) is
\beq \label{LHS1}
f_{X_{2:2}+X_3/3}(x) & = & \int_0^x f_{X_3/3}(y)f_{X_{2:2}}(x-y)\, dy \\
    & = & \int_0^x 3f(3y)\frac{d}{dx}[F^2(x-y)]\, dy \nonumber \\
    & = & 6\int_0^x f(3y)F(x-y)f(x-y)\, dy . \nonumber
\eeq
For the pdf of the right-hand side of (\ref{eqn_main1}), we have
\be \label{RHS1}
f_{X_{3:3}}(x) 
     =  3F^2(x)f(x)  
    =  6f(x)\int_0^x F(y)f(y)\, dy. 
\ee
Let $G(x):=F(x)f(x)$. It follows from (\ref{LHS1}) and (\ref{RHS1}) that (\ref{eqn_main1}) is equivalent to
\be \label{eqn1}
\int_0^x f(3y)G(x-y)\, dy = f(x)\int_0^x G(y)\, dy.
\ee
Differentiating  the left-hand side of (\ref{eqn1}) $n$ times with respect to $x$, we obtain
\[
\frac{d^n}{dx^n}\int_0^x f(3y)G(x-y)\, dy
  =  \sum_{i=1}^{n}f^{(n-i)}(3x)G^{(i-1)}(0)
       +  \int_0^x f(3y)G^{(n)}(x-y)\, dy.
\]
Applying the Leibniz rule for the $n$th derivative of a product of two functions to the right-hand side of (\ref{eqn1}), we obtain
\[
\frac{d^n}{dx^n}\left[f(x)\int_0^x G(y)\, dy\right]
  = \sum_{i=1}^n {n \choose i}f^{(n-i)}(x) G^{(i-1)}(x)+f^{(n)}(x)\int_0^x G(y)\, dy.
 \]
 In the last two equations letting $x=0$, we have 
\be \label{sum_eqn}
\sum_{i=1}^{n}3^{n-i}f^{(n-i)}(0)G^{(i-1)}(0)
        =
   \sum_{i=1}^n {n \choose i}f^{(n-i)}(0) G^{(i-1)}(0).
   \ee
Since $G(0)=0$ and $G'(0)=f^2(0)$, the above equation is equivalent to
\be \label{sum_eqn2}
\hspace{-0.5cm}\left[3^{n-2}-{n \choose 2}\right] f^{(n-2)}(0)f^2(0)= \sum_{i=3}^{n}\left[{n \choose i}-3^{n-i}\right]f^{(n-i)}(0)G^{(i-1)}(0),
\ee
where $n\ge 4$.
We will prove that (\ref{sum_eqn2}) implies (\ref{lemma}).
Equation (\ref{lemma}) is trivially true for $m=1$. 
To proceed by induction, assume (\ref{lemma}) holds true for all $1\le m\le n-3$, where $n\ge 4$. We need to prove it for $m=n-2$. Using the induction assumption, it is not difficult to obtain for $j=1,2,\ldots, n-2$
\[
G^{(j)}(0)  =  \sum_{i=0}^j{j \choose i}F^{(i)}(0)f^{(j-i)}(0)
     =  f^2(0)\left[\frac{f'(0)}{f(0)}\right]^{j-1}(2^j-1).
\]
Therefore, using the induction assumption again, we have for $i=3,4, \ldots, n-1$
\be \label{miss_11}
f^{(n-i)}(0)G^{(i-1)}(0)   
 =  \left[\frac{f'(0)}{f(0)}\right]^{n-3}f'(0)f^2(0)(2^{i-1}-1). 
\ee
Substituting this in the right-hand side of (\ref{sum_eqn2}) yields 
\[
\hspace{-0.3cm}\left[3^{n-2}-{n \choose 2}\right] f^{(n-2)}(0)
     =
    \left[\frac{f'(0)}{f(0)}\right]^{n-3}\! \! f'(0)\sum_{i=3}^{n}\left[{n \choose i}-3^{n-i}\right](2^{i-1}-1).
\]
To complete the proof of (\ref{lemma}), it is sufficient to show that
\[
3^{n-2}-{ n \choose 2}=\sum_{i=3}^{n}\left[{n \choose i}-3^{n-i}\right](2^{i-1}-1),
\]
which can be easily verified. This proves (\ref{lemma}). 
The claim in (i) follows from (\ref{lemma}) and the lemma.
\end{proof}

\begin{proof} (ii). Equation (\ref{eqn_main2}) is equivalent to
\be \label{e_main2}
6\int_0^zf(y)\int_0^{z-y}f(2x)f(3(z-y-x))\, dx \, dy = 6f(z)\int_0^z F(y)f(y)\, dy.
\ee
Denoting
\be \label{H}
 H(z-y):= \int_0^{z-y}f(2x)f(3(z-y-x))\, dx,
\ee
we write (\ref{e_main2}) as
\be \label{e_main3}
\int_0^z f(y)H(z-y)\, dy = f(z)\int_0^z G(y)\, dy.
\ee
Similarly to the proof of (i), differentiating $n$ times both sides of (\ref{e_main3}) with respect to $z$ and setting $z=0$, we have
\[
\sum_{i=1}^{n-1}f^{(n-1-i)}(0)H^{(i)}(0)=\sum_{i=1}^{n-1}{n \choose i+1}f^{(n-1-i)}(0)G^{(i)}(0).
\]
Since $H'(0)=G'(0)=f^2(0)$,  the last equation can be written for $k=n-1$ as
\be \label{new_main}
\hspace{-0.3cm} \left[1-{k+1 \choose 2}\right]f^{(k-1)}(0)f^2(0)=\sum_{i=2}^{k}\left[{k+1 \choose i+1}G^{(i)}(0)-H^{(i)}(0)\right]f^{(k-i)}(0).
\ee
Now we are in a position to prove (\ref{lemma}) by induction.
 (\ref{lemma}) holds true for $m=1,2,\ldots, k-2$.
Differentiating  (\ref{H}) 
with respect to $z$ and setting $z=y$, we have
\be \label{d_H}
H^{(n)}(0)  =  \sum_{i=1}^{n}2^{n-i}f^{(n-i)}(0)3^{i-1}f^{(i-1)}(0).
\ee
Under the induction assumption, (\ref{d_H}) implies for $j=1,2, \ldots, n-2$
\[
H^{(j)}(0)=
    \left[\frac{f'(0)}{f(0)}\right]^{j-1}f^2(0)\left(3^j-2^j\right).
\]
Using the induction assumption again, we have for $i=3,4, \ldots, n-1$
\[
f^{(n-i)}(0)H^{(i-1)}(0)=
\left[\frac{f'(0)}{f(0)}\right]^{n-3}f'(0)f^2(0)\left(3^{i-1}-2^{i-1}\right).
\]
Recalling (\ref{miss_11}) from the proof of (i), we rewrite (\ref{new_main}) as (note that $i=n$ corresponds to a 0 term)
\[
\left[1-{n \choose 2}\right]f^{(n-2)}(0)
=
\left[\frac{f'(0)}{f(0)}\right]^{n-3}f'(0)\sum_{i=3}^{n}\left[{n \choose i}(2^{i-1}-1)-(3^{i-1}-2^{i-1})\right]
\]
Thus, to prove (\ref{lemma}) for $k=n-2$ it is sufficient to show that
\[
1-{n \choose 2}=
\sum_{i=3}^{n}\left[{n \choose i}(2^{i-1}-1)-(3^{i-1}-2^{i-1})\right],
\]
which verifies. 
This proves (\ref{lemma}), which referring to the lemma, completes the proof of (ii).
\end{proof}

\section{Example}\label{example}

We will illustrate a possible application of Theorem~\ref{thm4} with an example (see also \cite{AV13}). Assume we have a simple random sample $X_1, X_2, \ldots, X_n$ for $n\ge 6$. Let us randomly divide the data set into six subsets, relabeled as  
\[
U_1, U_2,\ldots, U_{n/6}, \qquad V_1, V_2,\ldots, V_{n/6}, \qquad W_1, W_2,\ldots, W_{n/6}, 
\]
\[
X_1, X_2, \ldots, X_{n/6}, \qquad Y_1, Y_2,\ldots, Y_{n/6}, \qquad Z_1, Z_2, \ldots, Z_{n/6}.
\]
Define for $i=1,2,\ldots, n/4$
\[
R_i:= U_i+\frac{1}{2}V_i+\frac{1}{3}W_i, \quad      S_i:=\max\{U_i,V_i\}+\frac{1}{3}W_i \quad \mbox{and}\quad T_i:=\max\{X_i, Y_i, Z_i\}.
\]
Then, according to Theorem~\ref{thm4}, the $R$'s, the $S$'s, and the $T$'s will have a common distribution if and only if the original $X$'s follow an exponential distribution.  

Let us simulate a sample of size $n=180$ from a parent variable with $\exp (1)$ distribution. The values of $R_i$, $S_i$, and $T_i$ for $i=1,2,\ldots, 30$ are presented in Table~\ref{tab1}. 

\vspace*{5mm}
\begin{table}[h]
\centering
\begin{tabular}{|c|c|c|c|c|c|c|c|c|c|c|}
\hline
  R  & 3.56 & 0.70 & 0.62 & 3.33 & 0.30 & 0.78 & 2.29 & 0.97  & 1.59 & 0.50\\
\hline
   & 0.83 & 2.27 & 0.69 & 2.95 & 0.32 & 4.12 & 0.74 & 0.91 & 2.66 & 0.48\\
   \hline
   & 2.87 & 2.19 & 2.32 & 1.08 & 3.69 & 1.98 & 1.13 & 1.32 & 3.37 & 2.73\\
\hline
S & 2.98 & 1.23 &  0.77 & 2.75 & 0.44 & 0.75 & 1.97 & 1.08 & 1.43 & 0.50 \\
\hline
 & 0.76 & 1.65 & 0.58 & 2.39 & 0.27 & 3.41 & 0.73 & 0.89 & 2.63 & 0.40 \\
\hline
  & 2.22 & 1.87 & 4.25 & 1.07 & 2.72 & 1.74 & 1.07 & 1.11 & 2.71 & 3.87 \\
  \hline
T  & 2.07 & 0.60 & 0.97 & 0.47 & 2.84 & 0.84 & 1.02 & 1.84 & 0.57 & 2.88\\
\hline
 & 1.39 & 1.92  & 8.46 & 1.77 & 2.60 & 1.42 & 1.50 & 0.47 & 0.26 & 2.17 \\
\hline
 & 1.92 & 1.67  & 2.87 & 1.06 & 2.24 & 6.63 & 0.52 & 1.09 & 1.33  & 1.07 \\
\hline
\end{tabular}
\caption{Values $R_i$, $S_i$, and $T_i$ for $i=1,2,\ldots, 30$.
\label{tab1}}
\end{table}
\vspace*{5mm}

Using the non-parametric two-sample Wicoxon rank test, we compare the sample distribution functions of the $R$'s and $T$'s on one hand and the $S$'s and $T$'s on another. The test results provide evidence supporting an exponential underlying distribution. Namely, the hypothesis that the distributions of the $R$'s and the $T$'s are the same cannot be rejected with p-value 0.7635 (W=471). The hypothesis that the distributions of the $S$'s and the $T$'s are the same cannot be rejected with p-value 0.9357 (W=444).

\section{Concluding remarks}\label{sixth}
In this paper we proved characterizations of the exponential distribution conjectured by Arnold and Villase\~{n}or in \cite{AV13}. Furthermore, under the assumptions of Theorem~\ref{thm1} and using the same technique of proof, it can be seen that if $X_1+\frac{1}{2}X_2+\frac{1}{3}X_3$ has as its density any one of the following seven forms, then $X_i$'s are exponential:
\nbeq
3f(x)-6f(2x)+3\bar{F}(3x), & & 
3f(x)-6\bar{F}(2x)+3f(3x),  \\
3\bar{F}(x)-6f(2x)+3f(3x), & & 
3f(x)-6\bar{F}(2x)+3\bar{F}(3x),  \\
3F(x)-6f(2x)+3\bar{F}(3x), & & 
3\bar{F}(x)-6\bar{F}(2x)+3f(3x),  \\
3\bar{F}(x)-6\bar{F}(2x)+3\bar{F}(3x). & & 
\neeq
Likewise, under the assumptions of Theorem~\ref{thm2} and using the same technique of proof, it can be obtained that if $X_{3:3}$ has as its density any one of the preceding seven forms, then $X_i$'s are exponential. 

The results presented here can be extended in several directions. Naturally, one would like to explore the general case of samples of size $n$ for any $n\ge 4$. As we mentioned earlier, generalizations of Theorem~\ref{thm4} for arbitrary sample size are proved in \cite{YC16}. Here we would like to propose as open problems the following two characterizations, which would extend Theorem~\ref{thm1} and Theorem~\ref{thm2}, respectively.

 \begin{proposition} Let $X_1, X_2, \ldots, X_n$ be i.i.d. random variables, where $n\ge 4$. Assume $\varphi(t)$ is finite for all $t$ in a neighbourhood of zero. If for $x>0$
\[
\sum_{j=1}^n \frac{1}{j}X_j\quad \mbox{has pdf}\quad \sum_{j=1}^n { n \choose j}(-1)^{j-1}jf(jx),
\]
then $X_1\sim \exp(\lambda)$ for some $\lambda>0$.
\end{proposition}

 \begin{proposition}  Let $X_1, X_2, \ldots, X_n$ be i.i.d. random variables, where $n\ge 4$. Assume the cdf  $F$ has a power series representation in a neighborhood of zero. If for $x>0$ 
\[
X_{n:n} \quad \mbox{has pdf}\quad \sum_{j=1}^n { n \choose j}(-1)^{j-1}j\bar{F}(jx),
\]
then $X_1\sim \exp(1)$.
\end{proposition}

\begin{acknowledgments}
This work has been supported by the National Scientific Foundation of Bulgaria at the Ministry of Education and Science, grant No KP-6-H22/3. The author acknowledges the valuable suggestions from the referees and the editor.
\end{acknowledgments}

\end{document}